\documentclass[preprint]{amsart}
\usepackage{color}
\usepackage{xifthen}
\newboolean{pedro}
\setboolean{pedro}{true}
\newcommand{\pedro}{\ifthenelse{\boolean{pedro}}{\color{black}
    \setboolean{pedro}{false}}{\color{black}\setboolean{pedro}{true}}}
\usepackage{amssymb}
\usepackage{enumerate}
\newtheorem{theorem}{Theorem}
\newtheorem*{theorem*}{Theorem}
\newtheorem{lemma}{Lemma}
\newtheorem*{lemma*}{Lemma}
\newtheorem{definition}{Definition}

\newcommand{\hot}{\ensuremath{{h.o.t.}}}

\newcommand{\cc}{\ensuremath{(\mathbb{C}^{2},0)}}
\newcommand{\ord}{\ensuremath{\mathrm{ord}}}

\title{Vector flows and the analytic moduli of singular plane branches}
\author{P. Fortuny Ayuso}
\email{fortunypedro@uniovi.es}
\address{Dpt. of Mathematics, University of Oviedo, Spain}
\date{\today}
\dedicatory{To Prof. Felipe Cano on his sixtieth birthday.}
\subjclass[2010]{32S05, 32S65, 14H20}
\begin{document}

\begin{abstract}
  We provide a geometric elementary proof of the fact that an analytic plane branch is analytically equivalent to one whose terms corresponding to contacts with holomorphic one-forms ---except for Zarkiski's $\lambda$-invariant--- are zero (so called ``short parametrizations''). This is the main step missed by Zariski in his attempt to solve the moduli problem.
\end{abstract}
\maketitle
 
\section{Introduction and notation}
The classification of germs of irreducible analytic plane curves (usually called
branches) $\Gamma\subset \cc$ under analytic equivalence was an open problem
until \cite{Hefez-Hernandes-classification}, where a complete description of the
moduli of $\Gamma$ is given in terms of what the authors call a \emph{normal
  form}. Since Zariski's monograph \cite{Zariski4} (whose English translation we use), little was achieved in
terms of finding new analytic invariants before the work of Hefez and
Hernandes. Their classification result uses
mainly algebraic tools and is based on the Complete Transversal Theorem of
\cite{bruce-kirk-dupleiss}.

Let $\Gamma\subset\cc$ be an analytic branch. It is well-known (see Theorem
2.2.6 of \cite{Wall}, for instance) that it
admits a Puiseux parametrization $\phi(t)$:
\begin{equation}\label{eq:puiseux}
  \Gamma=\phi(t)=(x(t),y(t))=\bigg(
    t^n, t^m + \sum_{i>m} a_it^i
  \bigg),
\end{equation}
where $n\leq m$ and $\phi:(\mathbb{C},0)\rightarrow\cc$ is an injective
map. Moreover, after possibly an algebraic change of variables of the form
$y=y-s(x)$, we may assume that $n$ does not divide $m$ (which we shall write
$n\nmid m$). Zariski, in \cite{Zariski4}, tries to reduce the analytic classification
to finding the simplest parametrization like \eqref{eq:puiseux}, \emph{i.e.} 
having as many coefficients $a_i$ equal to $0$ as possible. He finds several
conditions which allow him to compute the moduli in some elementary cases albeit in a
rather convoluted way. As a matter of fact, the classification result of
\cite{Hefez-Hernandes-classification} consists in completely describing those
``simple'' parametrizations, which they aptly call ``normal forms''. Before
proceeding any further, we introduce the basic notation.

Given $\Gamma$, we may assume after an analytic change of coordinates that it
has a parametrization \eqref{eq:puiseux}. The \emph{semigroup} $S_{\Gamma}$ of
$\Gamma$ is the additive subsemigroup of $\mathbb{N}$:
\begin{equation*}
  S_{\Gamma} = \left\{
    \ord_t f(\phi(t)): f(x,y)\in \mathbb{C}\left\{ x,y \right\}
  \right\}
\end{equation*}
(which order is, for each $f(x,y)$, the intersection multiplicity of the curve
$\Delta\equiv (f(x,y)=0)$ with $\Gamma$). The \emph{conductor} of $S_{\Gamma}$
is the least $c\in S_{\Gamma}$ such that any $k\in \mathbb{N}$ greater than or
equal to $c$ belongs to $S_{\Gamma}$ (such a $c$ is guaranteed to exist). The
number $c$ satisfies
\begin{equation*}
  c = \dim_{\mathbb{C}}\mathbb{C}\left\{ t \right\}/J
\end{equation*}
where $J$ is the conductor ideal of $\mathbb{C}\left\{ x,y \right\}/(f(x,y))$ in
$\mathbb{C}\left\{ t \right\}$ via the parametrization given above (for the
definition of the conductor of a semigroup and the previous properties, see
\cite{Casas} Prop. 5.8.6 and the paragraph before, for instance).

Zariski points out that one can easily eliminate all the terms belonging to the
semigroup from a Puiseux expansion. But the first strictly analytic invariant
found by him is now called the $\lambda$-invariant:
\begin{theorem*}[\cite{Zariski-1966}, see also \cite{Zariski4} pp. 22, 23]
  Either $\Gamma$ is analytically equivalent to $(t^n, t^m)$ or there is
  $\lambda\in \mathbb{N}$ with $m<\lambda<c$ such that $\Gamma$ is analytically
  equivalent to
  \begin{equation*}
    \bigg(t^n,
    t^m + t^{\lambda} + \sum_{i>\lambda}^{c-1}\overline{a}_it^k
    \bigg)
  \end{equation*}
  where $\overline{a}_i=0$ if $i\in S_{\Gamma}$ or $i+n=pm$ for $p>1$.
\end{theorem*}
Even more; let $\Lambda$ be the set of contact orders of holomorphic
differentials with the curve $\Gamma$ (this set is implicitly used by Zariski in
\cite{Zariski-1966} and \cite{Zariski4} but a formal definition first appeared,
as far as we are aware, in \cite{Hefez-Hernandes-classification} pg. 291):
\begin{equation*}
  \Lambda = \left\{
    \ord_t(\phi^{\ast}\omega(x,y))+1: \omega(x,y)=A(x,y)dx+B(x,y)dy
  \right\}
\end{equation*}
(the $+1$ is added for simplicity: notice that the $dx$ in $\omega(x,y)$
provides a factor $nt^{n-1}dt$, so that $\ord_t\phi^{\ast}\omega$ is always at
least $n-1$). With this definition, Zariski's invariant is just
\begin{equation*}
  \lambda = \min\{i\in \Lambda, i\not\in S_{\Gamma}\}-n,
\end{equation*}
and $\lambda=\infty$ if and only if $\Gamma$ is analytically equivalent to the
cusp $(t^n, t^m)$ (these are the main results of \cite{Zariski-1966}).  From
this starting point, Zariski in \cite{Zariski4} tried (implicitly, as he never
used $\Lambda$ explicitly) to relate the set $\Lambda$ and the analytic moduli
of $\Gamma$ without success. The classification result of
\cite{Hefez-Hernandes-classification} shows how $\Lambda$ is, to all extents,
the \emph{main} analytic invariant:
\begin{theorem*}[\cite{Hefez-Hernandes-classification}, Theorem 2.1]
  The branch $\Gamma$ is either analytically equivalent to $(t^n, t^m)$ or
  $\lambda<c$ and it is equivalent to a parametrization (normal form)
  \begin{equation*}
    \bigg(
    t^n, t^{m} + t^{\lambda} + \sum_{\substack{i>\lambda\\ i+n\not\in \Lambda}}
    \tilde{a}_{i}t^i
    \bigg).
  \end{equation*}
  Moreover, two such parametrizations (with $\tilde{a}_i$ and
  $\tilde{a}_i^{\prime}$, respectively) corresponding to branches with same
  semigroup and same set of contacts $\Lambda$, are equivalent if and only if
  there is $r\in \mathbb{C}^{\ast}$ with $r^{\lambda-m}=1$ and
  $\tilde{a}_i=r^{i-m}\tilde{a}^{\prime}_i$.
\end{theorem*}
The second part of the theorem (once the normal form has been computed) can be
said to have been known by Zariski (after a careful reading of
\cite{Zariski4}). The key result is, thus, the \emph{elimination} from a Puiseux
expansion of all the terms $a_i$ such that $i>\lambda$ and $i+n\in
\Lambda$. Notice that there are a finite number of nonzero $\tilde{a}_{i}$: if
$i\geq c$ then $\tilde{a}_i=0$.

The aim of this paper is to provide a geometric ---dynamic--- proof of this
elimination step using the fact that \emph{differential forms arise from
  differential equations}. This simple fact, overlooked by Zariski, provides a
natural and elementary argument which essentially solves the moduli problem.

\section{Contact Transfer. Taylor expansions.}
The main tool we shall use in the rest of the paper is the ``transfer of contact'' between two parametric plane vectors which are colinear up to some order. Two plane vector fields $u=(a,b)$, $v=(p,q)$ in $\mathbb{C}^2$ are colinear if and only if their determinant is zero:
\begin{equation*}
  \left|
    \begin{matrix}
      a & b\\
      p & q
    \end{matrix}
  \right| = aq - bp = 0.
\end{equation*}
When the components $a,b,p,q$ are functions of one parameter $t$, we can speak of \emph{colinearity to some order}:
\begin{definition}
  Let $a(t),b(t),p(t),q(t)$ be holomorphic functions at $0\in \mathbb{C}$. We say that $u(t)=(a(t),b(t))$ and $v(t)=(p(t),q(t))$ are \emph{colinear to order $j$} if
  \begin{equation*}
    a(t)q(t) - b(t)p(t) = \left|
      \begin{matrix}
        a(t) & b(t)\\
        p(t) & q(t)
      \end{matrix}
    \right| = t^jd(t)
  \end{equation*}
  for some holomorphic function $d(t)$ (notice that $d(0)$ may be $0$). The value $d(0)$ will be called the \emph{contact coefficient}.
\end{definition}

The importance of this concept comes from the following simple result, to which we shall refer as the ``Contact Transfer Lemma'': if two parametric vectors are colinear up to some order then they ``transfer'' their contact with linear forms (notice that the contact \emph{is not the same}, just related):

\begin{lemma}[Contact Transfer]\label{lem:contact-transfer}
  Let $u(t)=(a(t),b(t))$ and $v(t)=(p(t),q(t))$ be two parametric vectors whose components are holomorphic at $0\in \mathbb{C}$, colinear to order $j$, with $p(t)\neq 0$. Let $A(t)$ and $B(t)$ be holomorphic functions at $0\in\mathbb{C}$. Then
    \begin{equation*}
      A(t)a(t)+B(t)b(t) = \frac{a(t)}{p(t)}(A(t)p(t)+B(t)q(t)) +
      \frac{B(t)}{p(t)}t^j \overline{d}(t) 
    \end{equation*}
    for some function $\overline{d}(t)$ holomorphic at $0\in\mathbb{C}$.
\end{lemma}
\begin{proof}
  By definition, as $p(t)\neq 0$:
  \begin{equation*}
    b(t) = \frac{a(t)}{p(t)}q(t)-\frac{1}{p(t)}t^jd(t)
  \end{equation*}
  for some $d(t)$ holomorphic at $0$. By direct substitution:
  \begin{equation*}
    A(t)a(t)+B(t)b(t) = a(t)A(t)
    + B(t)\left(\frac{a(t)}{p(t)}q(t)
    - \frac{1}{p(t)}t^j d(t)\right)
\end{equation*}
taking common factor and distributing the parenthesis, we obtain
\begin{equation*}
  A(t)a(t)+B(t)b(t) = \frac{a(t)}{p(t)}\left( A(t)p(t) + B(t)q(t) \right) -
  \frac{B(t)}{p(t)}t^j d(t) 
\end{equation*}
as desired.
\end{proof}
We shall repeatedly apply the Contact Transfer Lemma when $(a(t),b(t))$ is colinear to order $j+n-1$ to the tangent vector of $\Gamma=(t^n, t^m + \hot)$, say $(\dot{x}(t),\dot{y}(t))$, to obtain, for some specific holomorphic functions $F_i(x,y)$
\begin{equation*}
  \begin{split}
  \frac{\partial F_i}{\partial x}(x(t),&y(t))a(t) + \frac{\partial F_{i}}{\partial y}(x(t),y(t))b(t) =\\
  \frac{a(t)}{\dot{x}(t)}&\left(
    \frac{\partial F_{i}}{\partial x}(x(t),y(t))\dot{x}(t) + \frac{\partial F_{i}}{\partial y}(x(t),y(t))\dot{y}(t)
  \right) + \frac{\partial F_{i}}{\partial y}t^jd(t)=\\
  &\frac{a(t)}{\dot{x}(t)}\frac{dF_{i}(x(t),y(t))}{dt}+
  \frac{\partial F_{i}}{\partial y}t^jd(t)
\end{split}
\end{equation*}
inside a Taylor expansion depending on a parameter $t$. As the reader will have noticed, we set $A(t)=(\partial F_i/\partial x)(x(t),y(t))$ and $B(t)=(\partial F_i/\partial y)(x(t),y(t))$. However, to be precise, we need to introduce the required notation. Consider the ODE
\begin{equation}
  E \equiv
  \label{eq:ode-1}
      \left\{
      \begin{array}{l}
        \dot{x} = X_1(x,y)\\
        \dot{y} = Y_1(x,y)
      \end{array}
    \right.
\end{equation}
where $X_1(x,y)$ and $Y_1(x,y)$ are holomorphic functions defined in an open set $(0,0)\in U\subset \mathbb{C}^2$. Let $(x_0,y_0)\in U$ be an initial condition. A solution of \eqref{eq:ode-1} with initial condition $(x_0,y_0)$ is a holomorphic function $s\mapsto (x(s),y(s))$ defined in a neighbourhood $V$ of $0\in \mathbb{C}$ such that $(x(0),y(0))=(x_0,y_0)$ and for any $s\in V$,
\begin{equation}
  \label{eq:ode-solution-1}
      \left\{
      \begin{array}{l}
        \dot{x}(s) = X_1(x(s),y(s))\\
        \dot{y}(s) = Y_1(x(s),y(s))
      \end{array}
    \right.
\end{equation}
Assume $X_k$ and $Y_k$ have been defined for all $k \leq i$. Set, inductively,
\begin{equation}
  \label{eq:ode-ith-term}
  \begin{split}
    X_{i+1}(x,y) &:= \frac{\partial X_{i}(x,y)}{\partial x}X_1(x,y) + \frac{\partial X_i(x,y)}{\partial y}Y_1(x,y).\\
    Y_{i+1}(x,y) &:= \frac{\partial Y_{i}(x,y)}{\partial x}X_1(x,y) + \frac{\partial Y_i(x,y)}{\partial y}Y_1(x,y).    
    \end{split}
\end{equation}
By the chain rule, a solution $(x(s),y(s))$ as above satisfies, by \eqref{eq:ode-1} and \eqref{eq:ode-solution-1},
\begin{equation}
  \label{eq:chain-rule-1}
  \frac{d^ix}{ds^i}(0) = X_i(x_0,y_0), \,\,\,\frac{d^iy}{ds^i}(s) = Y_i(x_0,y_0).
\end{equation}
for all $i\geq 1$. Thus, by Taylor's Theorem, $x(s)$ and $y(s)$ have the following convergent expansions:
\begin{equation}
  \label{eq:taylor-expansion-solution-1}
  \begin{split}
  x(s) &= x_0 + X_1(x_0,y_0)s + \sum_{i>1} X_i(x_0,y_0) \frac{s^i}{i!}\\
  y(s) &= y_0 + Y_1(x_0,y_0)s + \sum_{i>1} Y_i(x_0,y_0) \frac{s^i}{i!}.
\end{split}
\end{equation}

In this paper we are concerned with a germ of analytic curve and a differential equation which ``almost'' leaves it invariant. Consider an analytic curve $\Gamma\equiv (x_0(t),y_0(t))$ with $x_0(0)=y_0(0)=0$, of the form $x_0(t)=t^n$, $y_0(t)=t^m +\hot$ for $m>n>0$ (the rest of the expansion is irrelevant for now; here and in the rest of the paper, $\hot$ means ``terms of order higher than the previous one'').  We use the indices $x_0(t)$, $y_0(t)$ to emphasize that we shall consider the points of $\Gamma$ as (parametric) initial conditions. Let $f(x,y)$ be an irreducible holomorphic function at $(0,0)$ such that $f(x_0(t),y_0(t))=0$. From this equality, we get, by differentiation,
\begin{equation}
  \label{eq:f-equals-0}
  f_x(x_0(t),y_0(t))\dot{x}_0(t) + f_{{y}}(x_0(t),y_0(t))\dot{y}_0(t) = 0
\end{equation}
where $f_x$ and $f_y$ denote the $x$ and $y$ partial derivatives. The
irreducibility of $f(x,y)$ implies that $f_y(x_0(t),y_0(t))\neq 0$ for
$t\neq 0$. This gives
\begin{equation}
  \label{eq:quotient-of-partials-of-f}
  \frac{f_x(x_0(t),y_0(t))}{f_y(x_0(t),y_0(t))} =
  -\frac{\dot{y}_0(t)}{\dot{x}_0(t)}
\end{equation}
(which makes sense because $x_0(t)=t^n\neq 0$ for $n>0$).

Consider the differential equations
  \begin{equation}
    \label{eq:two-diffeqs}
    \begin{array}{ll}
    E\equiv \left\{
      \begin{array}{l}
        \dot{x} = X_{1}(x,y)\\
        \dot{y} = Y_{1}(x,y)
      \end{array}
    \right.
    &
    \overline{E}\equiv
    \left\{
      \begin{array}{l}
        \dot{x} = f_{{y}}\left(X_1/f_{{y}}\right)(x,y) = \overline{X}_1(x,y) =X_1(x,y)\\
        \dot{y} = -f_x \left(X_1/f_{{y}}\right)(x,y) = \overline{Y}_1(x,y)
      \end{array}      
    \right.
    \end{array}
  \end{equation}
  and let $(t,s)\mapsto (x(t,s),y(t,s))$ and $(t,s)\mapsto (\overline{x}(t,s),\overline{y}(t,s))$ be their respective solutions for the initial conditions $(x_0(t),y_0(t))$. All the functions $x(t,s), y(t,s), \overline{x}(t,s)$ and $\overline{y}(t,s)$ are holomorphic and defined in a neighbourhood of $(0,0)\in \mathbb{C}^{2}$, by the analytic dependence of the solutions of an ODE on the initial conditions (Theorem 1.1 of  \cite{ilyashenko-yakovenko-lectures}). Notice that, by definition of differential equation, the points $(\overline{x}(t,s),\overline{y}(t,s))$ are all in $\Gamma$ (i.e. $\Gamma$ is invariant by $\overline{E}$), as $\overline{E}$ represents a vector field tangent to the set $\Gamma\equiv f(x,y)=0$ at all the points of $\Gamma$.

  \textbf{Further Assumptions.} From now on, we assume that $Y_{1}(x,y)\in(x,y)^2$ and $\ord_xX_{1}(x,0)\geq 2$. We also impose that, if $a(t)=X_1(x_0(t),y_0(t))$ and $b(t)=Y_1(x_0(t),y_0(t))$, then $(a(t),b(t))$ and $(\dot{x}_0(t),\dot{y}_0(t))$ are colinear to order $j+n-1$ for some $j>m$, \emph{with nonzero contact coefficient}.

The last condition is exactly what will imply that the flow associated to $E$ leaves $\Gamma$ ``almost'' invariant (to order $j$), affecting the $j-$th component of the $y$ coordinate \emph{linearly}. This is the content of the next result.
  
  \begin{lemma}\label{lem:solutions-are-approximate}
 With the notations and hypotheses of this section, there exist holomorphic functions $\tilde{x}(t,s)$ and $\tilde{y}(t,s)$ in a neighbourhood of $(0,0)\in \mathbb{C}^{2}$ such that
  \begin{equation}
    \label{eq:order-error-solution-lemma}
    x(t,s) = \overline{x}(t,s) + t^js^{2}\tilde{x}(t,s),\,\,\,
    y(t,s) = \overline{y}(t,s) + \alpha t^js + t^{j+1}s\tilde{y}(t,s)
  \end{equation}
  for some nonzero $\alpha\in \mathbb{C}$.
\end{lemma}
\begin{proof}
  Rewrite \eqref{eq:taylor-expansion-solution-1} for all the points
  $(x_0(t),y_0(t))\in \Gamma$ in a sufficiently small neighbourhood of $(0,0)$ as
  \begin{equation}
  \label{eq:taylor-expansion-solution-2}
  \begin{split}
    x(t,s) &= x_0(t) + X_1(x_0(t),y_0(t))s +
    \sum_{i>1} X_i(x_0(t),y_0(t)) \frac{s^i}{i!}\\
    y(t,s) &= y_0(t) + Y_1(x_0(t),y_0(t)) s +
    \sum_{i>1} Y_i(x_0(t),y_0(t)) \frac{s^i}{i!}.
\end{split}
\end{equation}
and the corresponding equalities for $\overline{x}(t,s), \overline{y}(t,s)$ and $\overline{X}_i, \overline{Y}_i$. These equalities hold (and the series are convergent) by the analytic dependence of the solutions of an ODE on the initial conditions. We also have the equivalents of \eqref{eq:chain-rule-1} for $(x(t,s),y(t,s))$:
\begin{equation}
  \label{eq:chain-rule-2}
  \frac{\partial^ix}{\partial s^i}(t,0) = X_i(x_0(t),y_0(t)),
  \,\,\,\frac{\partial^iy}{\partial s^i}(t,0) = Y_i(x_0(t),y_0(t)).
\end{equation}
and for $\overline{x}(t,s), \overline{y}(t,s)$:
\begin{equation}
  \label{eq:chain-rule-2}
  \frac{\partial^i\overline{x}}{\partial s^i}(t,s) =
  \overline{X}_i(x_0(t),y_0(t)),
  \,\,\,\frac{\partial^i\overline{y}}{\partial s^i}(t,s) =
  \overline{Y}_i(x_0(t),y_0(t)).
\end{equation}
Before proceeding any further, notice that $Y_1(x,y)\in(x,y)^2$ by hypothesis. Assume that $Y_k(x,y)\in(x,y)^{2}$ for $k\leq i$. Because $X_1(x,y),Y_1(x,y)\in (x,y)$ and
\begin{equation*}
Y_{i+1}= \frac{\partial Y_i}{\partial x}X_1(x,y) + \frac{\partial Y_i}{\partial y}Y_1(x,y),
\end{equation*}
we conclude that $Y_{i+1}(x,y)\in(x,y)^2$ for all $i$.
  
Using \eqref{eq:taylor-expansion-solution-2}, the lemma is proved if we show that
\begin{equation*}
  \begin{split}
  X_i(x_0(t),y_0(t)) &= \overline{X}_i(x_0(t),y_0(t)) + t^j\tilde{x}_i(t),\\
  Y_i(x_0(t),y_0(t)) &= \overline{Y}_i(x_0(t),y_0(t)) + t^{j}\tilde{y}_i(t)
\end{split}
\end{equation*}
for $\tilde{x}_i(t),\tilde{y}_i(t)$ holomorphic at $0\in \mathbb{C}$, with $\tilde{x}_1(t)=0$,  $\tilde{y}_1(0)\neq 0, \tilde{y}_i(0)=0$ for $i>1$. This is true certainly for $i=1$: on one hand, $X_1=\overline{X}_1$; on the other, by \eqref{eq:two-diffeqs} we have:
\begin{equation*}
  \overline{Y}_1(x_0(t),y_0(t)) = -
  \frac{f_x(x_0(t),y_0(t))}{f_y(x_0(t),y_0(t))}X_1(x_0(t),y_0(t))
  = \frac{\dot{y}_0(t)}{\dot{x}_0(t)}X_{1}(x_0(t),y_0(t)),
\end{equation*}
and by the hypothesis on the colinearity of $(X_1(x_0(t),y_0(t)),Y_1(x_0(t),y_0(t)))$ and $(\dot{x}_0(t),\dot{y}_0(t))$ to order $j+n-1$, we know that:
\begin{equation*}
  \dot{y}_0(t)X_1(x_0(t),y_0(t)) - \dot{x}_0(t)Y_1(x_0(t),y_0(t)) = t^{j+n-1}d(t)
\end{equation*}
with $\alpha:=d(0)\neq 0$. Now,  by substitution:
\begin{equation*}
  \overline{Y}_1(x_0(t),y_0(t)) = 
  \frac{\dot{y}_0(t)}{\dot{x}_0(t)}X_{1}(x_0(t),y_0(t)) = Y_1(x_0(t),y_0(t)) + \frac{1}{n} t^j d(t)
\end{equation*}
with $d(0)\neq 0$, as desired.

Assume the results true for $k\leq i$ and consider the case $i+1$. By the Contact Transfer Lemma (Equality (1)) and the chain rule (Equality (2)), there is some $\kappa\in \mathbb{C}$ such that
\begin{equation*}
  \begin{split}
X_{i+1}(x_0(t),y_0(t)):=\frac{\partial X_i}{\partial x}X_1(x_0(t),y_0(t)) +
\frac{\partial X_i}{\partial y}Y_1(x_0(t),y_0(t)) \stackrel{(1)}{=}\\ \frac{X_1(x_0(t),y_0(t))}{\dot{x}_0(t)}\left( \frac{\partial X_i}{\partial x} \dot{x}_0(t) + \frac{\partial X_i}{\partial y}\dot{y}_0(t) \right) +
\frac{\partial X_i}{\partial y}(x_0(t),y_0(t))(\kappa t^j + \hot) \stackrel{(2)}{=}\\
\frac{X_1(x_0(t),y_0(t))}{\dot{x}_0(t)} \frac{dX_i(x_0(t),y_0(t))}{dt} +
\frac{\partial X_i}{\partial y}(x_0(t),y_0(t))  (\kappa t^j + \hot) = \star.
\end{split}
\end{equation*}
By the induction hypothesis, $X_i(x_0(t),y_0(t))=\overline{X}_i(x_0(t),y_0(t))+\overline{\kappa}t^j$ for some other $\overline{\kappa}\in \mathbb{C}$, hence
\begin{equation*}
  \begin{split}
  \star = \frac{X_1(x_0(t),y_0(t))}{\dot{x}_0(t)} \frac{d(\overline{X}_i(x_0(t),y_0(t)) + \overline{\kappa}t^j + \hot)}{dt} +\\
  \frac{\partial X_i}{\partial y}(x_0(t),y_{0}(t))(\kappa t^j + \hot)
\end{split}
\end{equation*}
and rewriting the derivative with respect to $t$ using the chain rule, we get
\begin{equation*}
  \begin{split}
  \star = \frac{X_1(x_0(t),y_0(t))}{\dot{x}_0(t)}\left(
    \frac{\partial \overline{X}_i}{\partial x}\dot{x}_0(t) +
    \frac{\partial \overline{X}_i}{\partial y}\dot{y}_0(t)
    + j\overline{\kappa }t^{j-1}+\hot\right) +\\
  \frac{\partial X_i}{\partial y}(x_0(t),y_0(t))(\kappa t^j + \hot).
\end{split}
\end{equation*}
The conditions $m>n$ and $\ord_x(X_1(x,0))\geq 2$ imply that $\ord_t(X_1(x_0(t),y_0(t)))>n-1$. Distributing the parenthesis and simplifying, using that $X_1=\overline{X}_1$, $\overline{Y}_1=-f_x\overline{X}_1/f_y$ and \eqref{eq:quotient-of-partials-of-f}, we conclude that:
\begin{equation*}
  \star = \frac{\partial \overline{X}_i}{\partial x}
  \overline{X}_1(x_0(t),y_0(t)) + \frac{\partial \overline{X}_i}{\partial y}
  \overline{Y}_1(x_0(t),y_0(t)) + t^j\tilde{x}_i(t)
\end{equation*}
for some holomorphic function $\tilde{x}_i(t)$ at $0$. This proves the result for $x(t,s)$.

For $y(t,s)$, the Contact Transfer Lemma and the chain rule give again:
\begin{equation*}
  \begin{split}
Y_{i+1}(x_0(t),y_0(t)):=\frac{\partial Y_i}{\partial x}X_1(x_0(t),y_0(t)) +
\frac{\partial Y_i}{\partial y}Y_1(x_0(t),y_0(t)) =\\ \frac{X_1(x_0(t),y_0(t))}{\dot{x}_0(t)}\left( \frac{\partial Y_i}{\partial x} \dot{x}_0(t) + \frac{\partial Y_i}{\partial y}\dot{y}_0(t) \right) +
\frac{\partial Y_i}{\partial y}(x_0(t),y_0(t)) (\kappa t^j + \hot) =\\
\frac{X_1(x_0(t),y_0(t))}{\dot{x}_0(t)} \frac{dY_i(x_0(t),y_0(t))}{dt} +
\frac{\partial Y_i}{\partial y}(x_0(t),y_{0}(t))(\kappa t^j + \hot) = \spadesuit,
\end{split}
\end{equation*}
for some $\kappa \in \mathbb{C}$.
Using the induction hypothesis, we get
\begin{equation*}
  \begin{split}
  \spadesuit = \frac{X_1(x_0(t),y_0(t))}{\dot{x}_0(t)} \frac{d(\overline{Y}_i(x_0(t),y_0(t)) + \overline{\kappa} t^{j} + \hot)}{dt} +\\
  \frac{\partial Y_i}{\partial y}(x_0(t),y_0(t)) (\kappa t^j + \hot)
\end{split}
\end{equation*}
for some $\overline{\kappa }\in \mathbb{C}$. Computing the derivative with respect to $t$ using the chain rule:
\begin{equation*}
  \begin{split}
  \spadesuit = \frac{X_1(x_0(t),y_0(t))}{\dot{x}_0(t)}\left(
    \frac{\partial \overline{Y}_i}{\partial x}\dot{x}_0(t) +
    \frac{\partial \overline{Y}_i}{\partial y}\dot{y}_0(t)
    + \overline{\kappa }t^{j-1}+\hot\right) +\\
  \frac{\partial Y_i}{\partial y}(x_0(t),y_0(t)) (\kappa t^j + \hot)
\end{split}
\end{equation*}
which, as $X_1=\overline{X}_1$, $\ord_x(X_1(x,0))\geq 2$, $Y_i\in(x,y)^2$ and $\overline{Y}_1 = -f_yX_1/ f_x$, implies
\begin{equation*}
  \spadesuit = \frac{\partial \overline{Y}_i}
  {\partial x}\overline{X}_1(x_0(t),y_0(t)) +
  \frac{\partial \overline{Y}_i}{\partial y}\overline{Y}_1
  (x_0(t),y_0(t))+ \tilde{\kappa }t^{j+1} + \hot
\end{equation*}
for some $\tilde{\kappa }\in \mathbb{C}$, as desired.
\end{proof}
We shall also need the following result:
\begin{lemma}\label{lem:solutions-order-n+1}
  With the hypothesis of the previous lemma,
  \begin{equation*}
    \ord_{t}(X_i(x_0(t),y_0(t))), \ord_t(\overline{X}_i(x_0(t),y_0(t))) > n.
  \end{equation*}
\end{lemma}
\begin{proof}
  By hypothesis, the result is true for $i=1$ as $\ord_x(X_1(x,0))\geq 2$ and
  $m>n>0$. By definition,
  \begin{equation*}
    X_{i+1} = \frac{\partial X_{i}}{\partial x}X_1(x,y) + \frac{\partial X_i}{\partial y}Y_1(x,y)
  \end{equation*}
  and the result follows by induction, as $X_1(x,0)\in(x)^2$, $Y_1(x,y)\in(x,y)^2$ and $m>n>0$ again. The same reasoning works for $\overline{X}_i$.
\end{proof}

\section{Reparametrization of Puiseux families}
With the notation and hypothesis of the previous section, we know that
\begin{equation}\label{eq:x-and-y-in-gamma-to-order-j}
  x(t,s) = \overline{x}(t,s) + t^js^{2}\tilde{x}(t,s),\,\,\,
  y(t,s) = \overline{y}(t,s) + \alpha t^js +  t^{j+1}s\tilde{y}(t,s)
\end{equation}
for some holomorphic functions $\tilde{x}(t,s)$ and $\tilde{y}(t,s)$ at $(0,0)\in \mathbb{C}^2$ and $0\neq\alpha\in \mathbb{C}$. Write an irreducible Puiseux expansion of $\Gamma$ as
\begin{equation}\label{eq:first-puiseux-expansion}
  \Gamma \equiv
  \left\{
    \begin{array}{l}
      x = t^n\\
      y = t^m + \sum_{i>m} a_it^i
    \end{array}
  \right.
\end{equation}
(recall that $m>n>0$). We already know that $(\overline{x}(t,s),\overline{y}(t,s))\in \Gamma$ for all $(t,s)$ in a neighbourhood of $(0,0)$, so that we should be able to ``rewrite'' $(\overline{x}(t,s),\overline{y}(t,s))$ as \eqref{eq:first-puiseux-expansion} somehow for each $s$. Indeed, by Lemma \ref{lem:solutions-order-n+1}, $\overline{X}_i(x_0(t),y_0(t))$ has order greater than $n$ for $i>0$. Then
\begin{equation}
  \label{eq:pre-puiseux-family}
  \left\{
    \begin{array}{l}
      \overline{x} (t,s) = t^n + \sum_{i>0}\overline{x}_i(t)s^i\\
      \overline{y}(t,s) = t^m + \sum_{i>0}\overline{y}_i(t)s^i
    \end{array}
  \right.
\end{equation}
for holomorphic functions $\overline{x}_i(t), \overline{y}_i(t)$ at $0\in\mathbb{C}$ with $\ord_t\overline{x}_i(t)>n$. This allows us to compute the $n-$th root of
\begin{equation*}
   1 + \sum_{i>0}\frac{\overline{x}_i(t)}{t^n}s^i=(\xi_n+st\tilde{u}(t,s))^n
 \end{equation*}
 where $\xi_n$ is an $n-$th root of unity and $\tilde{u}(t,s)$ is a holomorphic function at $(0,0)\in \mathbb{C}^2$. Thus, the function
 \begin{equation*}
   u(t,s) = t(\xi_n + st\tilde{u}(t,s))
 \end{equation*}
satisfies
 \begin{equation}
  \label{eq:first-change-of-vars}
  u(t,s)^n = \big(t ( \xi_n + st\tilde{u}(t,s) )\big)^n = t^n + \sum_{i>0}\overline{x}_i(t)s^i = \overline{x}(t,s)
\end{equation}
This defines a change of variables $\left\{u=u(t,s),s=s\right\}$
 whose inverse
 \begin{equation}
   \label{eq:change-of-vars-1}
   \left\{t(u,s) = u(\overline{\xi}_n + su \tilde{t}(u,s)), s=s\right\}
 \end{equation}
 (where $\tilde{t}(u,s)$ is a holomorphic function at $(0,0)\in \mathbb{C}^2$ and $\overline{\xi}_n$ is the complex conjugate of $\xi_n$) is holomorphic at $(0,0)$ and provides the desired equality:
 \begin{equation*}
  \label{eq:puiseux-family}
  \left\{
    \begin{array}{l}
      \overline{x} (u,s) = u^{n}\\
      \overline{y} (u,s) = u^m + \sum_{i>m}a_iu^i
    \end{array}
  \right.   
 \end{equation*}
valid for all $(u,s)$ in a neighbourhood of $(0,0)\in \mathbb{C}^2$. Notice that this is true for one of the $n-$th roots of unity $\xi_n$ (for the others, each $a_i$ is multiplied by $\tilde{\xi}_n^i$ for another root $\tilde{\xi}_{n}$). This  result \emph{does not mean} that $\Gamma$ is composed of fixed points of $\overline{E}$, as $u$ is not the initial parameter of \eqref{eq:first-puiseux-expansion}. It means that $\Gamma$ is \emph{invariant} by $\overline{E}$.

We now study $E$ in this new system of coordinates. Notice that the change of variables \eqref{eq:change-of-vars-1} satisfies, for all $k\in \mathbb{N}$:
\begin{equation*}
  t(u,s)^k = \big(u(\overline{\xi}_n+su\tilde{t}(u,s))\big)^k = u^k\overline{\xi}_n^k + su^{k+1}\underline{t}(u,s)
\end{equation*}
for some holomorphic function $\underline{t}(u,s)$ at $(0,0)\in \mathbb{C}^2$; this implies, by \eqref{eq:x-and-y-in-gamma-to-order-j}, as $j>m>n$, that the solution $(x(t,s),y(t,s))$ of $E$ has the expression
\begin{equation*}
  \left\{
    \begin{array}{l}
      x (u,s) = u^{n} + su^j(\underline{x}(u,s))\\
      y (u,s) = u^m + \sum_{i>m}a_iu^i + \alpha^{\prime}su^j +
      u^{j+1}s\underline{y}(u,s)
    \end{array}
  \right.   
\end{equation*}
for some $\alpha^{\prime} \neq 0$ and holomorphic functions $\underline{x}(u,s)$ and $\underline{y}(u,s)$ at $(0,0)\in \mathbb{C}^{2}$. Compute the $n-$th root
\begin{equation*}
  1 + su^{j-n}\underline{x}(u,s) = \left(\eta_n + su^{j-n}\tilde{v}(u,s)\right)^n
\end{equation*}
for some holomorphic function $\tilde{v}(u,s)$ at $(0,0)\in \mathbb{C}^2$.
Let
\begin{equation*}
  v(u,s) = u(\eta_n + su^{j-n}\tilde{v}(u,s)).
\end{equation*}
As above, the change of variables $\left\{ v=v(u,s), s = s \right\}$ has an inverse
\begin{equation*}
  \left\{ u(v,s) = v \left( \overline{\eta}_n + sv^{j-n}\underline{u}(v,s)\right),
    s = s  \right\}
\end{equation*}
which satisfies, for all $k\in \mathbb{N}$:
\begin{equation*}
  u(v,s)^k = v^k\overline{\eta}_n^k + sv^{j+(k-n)}u^{\prime}(v,s)
\end{equation*}
for some holomorphic $u^{\prime}(v,s)$. As $j>m>n$, we get, in the coordinates $(v,s)$ (again, for one of the $n-$th roots of unity $\eta_n$):
\begin{equation}
    \label{eq:reparametrization}
    \left\{
    \begin{array}{l}
      x (v,s) = v^{n}\\
      y (v,s) = v^m + \sum_{i>m}a_iv^i + \beta s v^j +
      v^{j+1}sy^{\prime}(v,s)
    \end{array}
  \right.   
\end{equation}
for some function $y^{\prime}(v,s)$ holomorphic in a neighbourhood of $(0,0)\in \mathbb{C}^2$ and $\beta\neq 0$.

As a consequence:
\begin{lemma}\label{lem:reparametrization}
  In the conditions of the previous section, there is an $s_j\in \mathbb{C}$ such that the curve $(x(t,s_{j}),y(t,s_j))$ admits a reparametrization
  \begin{equation*}
        \left\{
    \begin{array}{l}
      x (v,s_j) = v^{n}\\
      y (v,s_j) = v^m + \sum_{m<i<j}a_iv^i +
      v^{j+1}s_j y^{\prime}(v,s_j)
    \end{array}
  \right.   
\end{equation*}
that is, the term of order $j$ can be removed from the Puiseux expansion using the flow associated to a vector field, without modifying the previous ones.
\end{lemma}
\begin{proof}
  Take $s_j=-a_j/\beta$ in \eqref{eq:reparametrization}. Recall that the flow $\psi_s(x,y)$ is the map sending $(x,y)$ to the value at time $s$ of the solution of $E$ with initial condition $(x(0),y(0))=(x,y)$.

  Notice that $E$ is an autonomous system at a singular point, so that for any $M$, there is a neighbourhood $V$ of $t=0$ such that 
$(x(t,s),y(t,s))$ converges for $|s|<M$ and $t\in V$. (see \cite{ilyashenko-yakovenko-lectures}, Proposition 1.19, p. 12): hence, the argument works for whatever value $a_j$.
\end{proof}

\section{Elimination of terms using holomorphic vector fields}
Our main result is now a corollary of the previous ones.
\begin{theorem}
  Let $\Gamma$ have a Puiseux parametrization of the form \eqref{eq:puiseux}
  with $n<m$ and $n$ not dividing $m$. Then $\Gamma$ is analytically equivalent
  to $\Gamma^{\prime}$ with Puiseux parametrization
  \begin{equation*}
    \Gamma^{\prime} \equiv \bigg(
    t^n, t^m + a_{\lambda}^{\prime}t^{\lambda} +
    \sum_{\substack{i>\lambda\\ i+n\not\in\Lambda}}a^{\prime}_it^i
    \bigg).
  \end{equation*}
  Moreover, each elimination of a single term can be carried out by means of the
  time-$s$ flow associated to a holomorphic vector field (one for each term).
\end{theorem}
If $\lambda=\infty$, then $\Gamma$ is equivalent to $(t^n, t^m)$ and the
argument below works anyway. Notice how after the composition of a finite number
of flows, we have eliminated all the (removable) terms of order less than
$c$. The terms from $c$ on can be removed at once using a single diffeomorphism
(as shown in \cite{Zariski4}, Section 1 of Chapter III). Therefore, we shall
only deal with the elimination of the term of least order $j$ with
$j+n\in \Lambda$ and $a_j\neq 0$, provided the terms of lower order remain
unchanged.

\begin{proof}
  Take a parametrization of $\Gamma$ like \eqref{eq:puiseux} (recall that $n$ does not divide $m$):
  \begin{equation*}
    \Gamma=\phi(t)=(x(t),y(t))=\bigg(
    t^n, t^m + \sum_{i>m} a_it^i
    \bigg),
  \end{equation*}
  Let $j>m$ be the minimum integer such that $a_j\neq 0$, $j+n\in \Lambda$ and
  $j\neq \lambda$. By definition, there is a holomorphic differential form
  $\omega$ such that
    \begin{equation*}
    \omega=A(x,y)dx+B(x,y)dy,\;\mathrm{with}
    \;\;\ord_t(\phi^{\ast}\omega)=j+n-1.
  \end{equation*}
  We wish to apply Lemmas \ref{lem:solutions-are-approximate} and
  \ref{lem:reparametrization} using the differential equation (vector field) $E$ ``dual'' to $\omega$, and the corresponding $\overline{E}$ in \eqref{eq:two-diffeqs}:
  \begin{equation*}
    E\equiv \left\{
      \begin{array}{l}
        \dot{x} = B(x,y)\\
        \dot{y} = -A(x,y)
      \end{array}
    \right.\,\,\,
    \overline{E}\equiv \left\{
      \begin{array}{l}
        \dot{x} = f_y\left(B/f_y\right)(x,y) =B(x,y)\\
        \dot{y} = -f_x\left(B/f_y\right)(x,y)
      \end{array}
    \right.
  \end{equation*}
  so that we need to verify that $A(x,y)\in(x,y)^2$
  and $\ord_xB(x,0)\geq 2$. Notice that the colinearity condition between
  $u(t)=(B(x(t),y(t)),-A(x(t),y(t)))$ and
  $v(t)=(\dot{x}(t),\dot{y}(t))$ to order $j+n-1$ is provided by the
  contact between $\omega$ and $\Gamma$. There are two cases: $j<\lambda$ and
  $j>\lambda$.

  If $j<\lambda$ then, by definition of $\lambda$, we have
  $j\in S_{\Gamma}$. Either $j=s+n$ for $s\in S_{\Gamma}$ and we can take
  $A(x,y)=g(x,y)$ with $\ord_tg(x(t),y(t))=j$ and $B(x,y)=0$ (notice that
  $g(x,y)\in(x,y)^2$ because $j>m$ and $n$ does not divide $m$). Or (see
  \cite{Zariski-1966} pp. 785-786 or \cite{Zariski4} p. 23, last paragraph)
  we have $j+n=pm$ for $p>1$ and we can take $\omega=y^{p-1}dy$, for which
  $A(x,y)=0$ and $\ord_xB(x,0)\geq 2$.

  Assume that $j>\lambda$ and let $\omega$ be as above.
  Write $A(x,y)=a_{10}x + a_{01}y + \dots$ and
  $B(x,y)=b_{10}x + b_{01}y + \dots$. Substituting the parametrization of
  $\Gamma$ into $\omega$ gives
  \begin{equation*}
    \begin{split}
    \phi^{\ast}\omega = (&na_{10}t^{2n-1} + (na_{01} + mb_{10})t^{n+m-1}
    + (na_{01}+\lambda b_{10})a_{\lambda}t^{\lambda+n-1}
    +\hot\\
    + &mb_{01}t^{2m-1} + (m+\lambda)b_{01}a_{\lambda}t^{\lambda+m-1} + \hot)dt.
    \end{split}
  \end{equation*}
  If $a_{10}\neq 0$ then $j=n$, which is impossible. If the coefficient
  $(na_{01}+mb_{10})$ is not zero, then $\ord_t\phi^{\ast}\omega=n+m-1$ because
  $n$ does not divide $m$ (so that this term cannot be made zero either by any
  $a_{kl}x^ky^l$ or $b_{kl}x^ky^l$); but this would imply that $j=m<\lambda$,
  against the assumption. If now $a_{01}\neq 0$ then $mb_{10}=-na_{01}$ so that
  $b_{10}\neq 0$; however, in this case
  $(na_{01}+\lambda b_{10})a_{\lambda}\neq 0$ and one has the following
  possibilities:
  \begin{itemize}
  \item $2m-1<\lambda+n-1$, which would imply that $j=2m-n<\lambda$,
    against the assumption $j>\lambda$.
  \item $\lambda+n-1<2m-1$, which would imply that $j= \lambda$, against the
    same assumption.
  \item $\lambda+n-1=2m-1$, which would imply that $\lambda=2m-n$, which
    contradicts the definition of $\lambda$.
  \end{itemize}
  Hence, we must have $a_{10}=a_{01}=0$ and from this $b_{10}=0$, as otherwise
  $j=m<\lambda$. Thus, $A(x,y)\in(x,y)^2$ and
  $\ord_xB(x,0)\geq 2$.

  We have concluded, in any case, taking $X_1=B(x,y)$ and $Y_1=-A(x,y)$, that we
  are in the conditions of Lemmas \ref{lem:solutions-are-approximate} and
  \ref{lem:reparametrization}, and the result follows considering the flow
  $\psi_s(x,y)$ associated to $E$ (recall that this flow sends $(x,y)$ to the
  image at time $s$ of the solution of $E$ with initial condition $(x,y)$).
\end{proof}

Simply speaking, the flow $\psi_s$ corresponding to $E$ produces (after a
re\-pa\-ra\-metri\-zation) a translation proportional to $s$ in the $j$-th term
of the $y$-component of the Puiseux expansion of $\Gamma$ (and nothing before
that term), which permits the elimination of this term. The terms farther than
$j$ are modified holomorphically.
    
  The proof does not work for $j=\lambda$. Take $\omega = -mydx+nxdy$. One has:
  \begin{equation*}
    \lambda+n =
    \ord_t(\omega(t^n, t^m + t^{\lambda} + \sum_{i>\lambda} a_it^{i}))+1,
  \end{equation*}
  so that $\lambda+n$ is the contact of $\Gamma$ with a differential form of
  order $1$ on each component.\par
\vspace*{10pt}
\textbf{Acknowledgement:\/} The redaction of this paper has improved greatly thanks to an anonymous reviewer.


\end{document}